\documentclass[12pt]{amsart}
\usepackage[english]{babel}
\usepackage{amsmath,amsthm,amssymb,amsfonts}

\DeclareMathOperator{\supp}{\mathrm{supp}}

\newcommand{\N}{\mathbb{N}}
\newcommand{\R}{\mathbb{R}}
\newcommand{\C}{\mathbb{C}}

\DeclareMathOperator*{\esssup}{\mathrm{ess\,sup}}

\newtheorem{theorem}{Theorem}[section]
\newtheorem{proposition}[theorem]{Proposition}
\newtheorem{coro}[theorem]{Corollary}
\newtheorem{lemma}[theorem]{Lemma}

\begin{document}


\title[Sharp spectral multipliers for a new class of  Grushin type operators]
{Sharp spectral multipliers for a new class of  Grushin type  operators }

\author{Peng Chen}
\address{Peng Chen \\ Department of Mathematics \\ Macquarie University \\ NSW 2109 \\ Australia}
\email{achenpeng1981@163.com}

\author{Adam Sikora}
\address{Adam Sikora \\ Department of Mathematics \\ Macquarie University \\ NSW 2109 \\ Australia}
\email{sikora@maths.mq.edu.au}

\begin{abstract}  

We describe  weighted Plancherel estimates  and sharp  Hebisch-M\"uller-Stein type spectral multiplier 
result 
for a new class of Grushin  type operators. We also discuss the optimal exponent for Bochner-Riesz 
summability in this setting. 
\end{abstract}

\maketitle

\section{Introduction}

On the space $L^2(\R^{d_1}\times\R^{d_2}) $ with the standard Lebesgue
measure consider a class of Grushin type operators defined by the formula 
\begin{equation}\label{eq1}
L_\sigma=-\sum_{{{j}}=1}^{d_1}\partial_{x'_{{j}}}^2 -
\left(\sum_{{{j}}=1}^{d_1}|x'_{{j}}|^\sigma \right)
\sum_{{{k}}=1}^{d_2}\partial_{x''_{{k}}}^2
\end{equation}
where exponent $\sigma >0$. In the case $\sigma=2$, the spectral properties of these operators 
were studied by A. Martini and the second author  in  \cite{MaS12} where 
  sharp spectral multiplier  and optimal  Bochner-Riesz summability results were obtained.   The aim of this paper 
  is to obtain analogous results for the class of Grushin operators corresponding to the exponent $\sigma =1$. 
The general strategy of the proof of the sharp spectral multiplier result for  $\sigma =1$ is the same as  one 
described in \cite{MaS12} for $\sigma =2$. However, the proofs of two most crucial estimates (Proposition \ref{prop2.2}
and Lemma \ref{lem3.4} below) are new and technically significantly more difficult.
The spectral decompositions of the operators $L_2$ and $L_1$  are essentially  different.
We use results derived in  \cite{G} to obtain a description of
the spectral decomposition of the operator $L_1$  necessary  for the proof of Proposition \ref{prop2.2}
and Lemma \ref{lem3.4}.

 The closure of  operator $L_\sigma$, $\sigma >0$  initially defined on $C^\infty_c(\R^{d_1}\times\R^{d_2})$ is a non-negative self-adjoint operator and it admits a spectral resolution
$E_{L_\sigma}(\lambda)$ for all $\lambda \ge 0$, see \cite{RS}. By spectral theorem
for every
bounded Borel function $F : \R \to \C$, one can define the operator
$$
F(L_\sigma) = \int_\R F(\lambda) \,dE_{L_\sigma}(\lambda)
$$
which is bounded on $L^2(\R^{d_1}\times\R^{d_2})$. This paper is devoted to  spectral
multipliers  that is we investigate sufficient conditions on function $F$ under
which the operator $F(L_1)$ extends to bounded operator acting on spaces
$L^p(\R^{d_1}\times\R^{d_2})$ for some range of $p$. We also study closely related 
question of critical exponent $\kappa$ for which the Bochner-Riesz means $(1-t
L_1)_+^\kappa$ are bounded on $L^p(\R^{d_1}\times\R^{d_2})$ uniformly in $t \in [
0,\infty)$. In the sequel we shall only discuss the Grushin operator $L_1$ which for simplicity 
we denote just by $L$.

The motivation and rationale for spectral multiplier results of the type, which we consider here 
as well as relevant literature and earlier related multiplier results 
were described in details in the introduction to \cite{MaS12} and we refer readers to this paper 
for  in depth discussion. Here  we only want to briefly 
mention that the theory of spectral multipliers and Bochner-Riesz analysis are central part of 
harmonic analysis which have attracted a huge amount of attention,
see for example \cite{CS, DOS, heb93, MS94, SW} and references within. 
 One especially intriguing and   surprising  direction in the 
theory of spectral multipliers is devoted to investigation of sharp results for  sub-elliptic or degenerate 
operators. The main idea in this area is that the sharp results are expected to be determined by the Euclidean 
dimension of underling ambient space rather than the homogeneous dimension of  the space and corresponding 
heat semigroup. This part of  spectral multipliers theory was initiated by results obtained by 
W. Hebisch \cite{heb93} and D. M\"uller and E.M. Stein \cite{MS94}. Other examples of papers devoted to 
 sharp spectral multipliers  for  sub-elliptic or degenerate 
operators include  \cite{CP, CKS, CS, HZ, JST}.

\bigskip

Our two main results,  the sharp spectral multiplier and the corresponding optimal results for 
convergence of Bchner-Riesz means, are stated in Theorems~\ref{thm1m} and \ref{thm2r}
below. 
 Let $\eta$ be a non-trivial $C_c^\infty$ function with compact
support on $\R_+$. For function $F : \R \to \C$ we define $\delta_t F(x)=F(tx)$ and
set $D=\max\{d_1+d_2,3d_2/2\}$. By $W_2^s$ we denote $L^2$ Sobolev space that is 
$ \| F \|_{W_2^s}= \| (I-d_x^2)^{s/2}F  \|_2$.

\begin{theorem}\label{thm1m}
Suppose that function $F : \R \to \C$ satisfies
\begin{equation*}
\sup_{t>0} \|\eta \, \delta_t F \|_{W_2^s}  < \infty
\end{equation*}
for some $s > D/2$. Then the operator $F(L)$ is of weak type $(1,1)$ and bounded on
 $L^p(\R^{d_1}\times\R^{d_2})$ for all $p \in (1,\infty)$. In addition
\begin{equation*}
\|F(L)\|_{L^1 \to L^{1,w}} \leq C  \sup_{t>0} \|\eta \, \delta_t F
\|_{W_2^s} \quad and \quad 
 \qquad \|F(L)\|_{L^p \to L^{p}} \leq C_p \sup_{t>0} \|\eta \, \delta_t F\|_{W_2^s}.
\end{equation*}

\end{theorem}

The above result is sharp if $d_1 \ge d_2/2$, see discussion in Section \ref{sec5} below. A version of result essentially equivalent to Theorem~\ref{thm1m} 
can be expressed in terms of  Bochner-Riesz summability of the operator $L$. Our
approach allows us to obtain the following result which is again  optimal if
$d_1 \ge d_2/2$.

\begin{theorem}\label{thm2r}
Suppose that $\kappa > (D-1)/2$ and $p \in \left[
1,\infty\right]$. Then the Bochner-Riesz means $(1-t
L)_+^\kappa$ are bounded on $L^p(\R^{d_1}\times\R^{d_2})$ uniformly in $t \in [
0,\infty)$.
\end{theorem}

Proofs of Theorems~\ref{thm1m} and \ref{thm2r} are concluded in Section~\ref{sec4}. 
Similarly as in \cite{MaS12} the key point of proving Theorems
\ref{thm1m} and \ref{thm2r} is to obtain
``weighted Plancherel estimate'' for spectral multipliers of the considered
Grushin type operators. A proof of such estimates is described in Section~\ref{sec3} and constitutes 
a main original contribution of this paper to the discussed research area.  A part of  a proof of Theorems~\ref{thm1m} and \ref{thm2r} described in   Section~\ref{sec4} below
is essentially the same as in \cite{MaS12}. We repeat the short argument here for the sake of completeness. 
To make it easier to compare the results obtained in   \cite{MaS12}  and in this paper   
we try to  use the same notation as in  \cite{MaS12} whenever it is possible.

\section{Notation and preliminaries}\label{sec2}

A more general class of Grushin type operators which includes operators $L_\sigma$ 
for $\sigma>0$  defined above was studied in \cite{RS}. In what follows we will need the basic
results concerning the Riemannian distance corresponding to Grushin type operators
and the standard Gaussian bounds for the corresponding hear kernels which were obtain
in \cite{RS} and which we recall below.

\begin{proposition}
Let  $\rho$ be Riemannian distance corresponding to  the Grushin operator $L$ and let
$B(x,r)$ be the ball with centre at $x$ and radius $r$.
Then
\begin{equation}\label{eqCD}
\rho(x,y) \sim |x' - y'| + \begin{cases}
\frac{|x''-y''|}{(|x'| + |y'|)^{1/2}} &\text{if $|x''-y''| \leq (|x'| + |y'|)^{3/2}$,}\\
|x''-y''|^{2/3} &\text{if $|x''-y''| \geq (|x'| + |y'|)^{3/2}$.}
\end{cases}
\end{equation}
Moreover the volume of $B(x,r)$ satisfies following estimates
\begin{equation}\label{eqVr}
|B(x,r)| \sim r^{d_1+d_2} \max\{r,|x'|\}^{d_2/2},
\end{equation}
and in particular, for all $\lambda \geq 0$,
\begin{equation}\label{doubling}
|B(x,\lambda r)| \leq C (1+\lambda)^Q |B(x,r)|
\end{equation}
where $Q=d_1+\frac{3d_2}{2}$ is a homogenous dimension of the considered metric space.
Next, there exist constants $b,C > 0$ such that, for all $t > 0$, the integral kernel $p_t$ of the operator $\exp(-tL)$ satisfies  the following Gaussian bounds
\begin{equation}\label{Gauss}
|p_t(x,y)| \leq C |B(y,t^{1/2})|^{-1} e^{-b \rho(x,y)^2/t}
\end{equation}
for all $x,y \in \R^{d_1}\times\R^{d_2}$.
\end{proposition}
\begin{proof} For the proof, we refer readers to
\cite[Proposition 5.1 and Corollary 6.6]{RS}.
\end{proof}

\bigskip

Next,  let
$\mathcal{F} : L^2(\R^{d_1} \times \R^{d_2}) \to L^2(\R^{d_1} \times
\R^{d_2})$ be the partial Fourier transform in variables $x''$
defined by
\[
\mathcal{F} \phi(x',\xi) = (2\pi)^{-d_2/2} \int_{\R^{d_2}}
\phi(x',x'') \, e^{-i \xi \cdot x''} \, dx''.
\]
 Then
\begin{eqnarray*}
\mathcal{F} L\phi (x',\xi) = \widetilde{L}_{\xi} \, \mathcal{F}
\phi(x',\xi)
\end{eqnarray*}
where $\widetilde{L}_{\xi}$ is Schr\"odinger type operators defined by 
$$
\widetilde{L}_{\xi}=-\Delta_{d_1}+\left(\sum_{{{j}}=1}^{d_1}|x'_{{j}}|\right)|\xi|^2
$$
acting on $L^2(\R^{d_1})$ where $\xi\in \R^{d_2}$. In what fallows
we will need the following estimates for the operator
$\widetilde{L}_{\xi}$, compare  \cite{heb93, CS, CKS} and \cite{MaS12}.
\begin{proposition}\label{prop2.2}
For all $\gamma \in [ 0,\infty)$ and $f \in L^2(\R^{d_1})$,
\begin{equation}\label{eq12}
\|\, (\sum_{{{j}}=1}^{d_1}|x'_{{j}}|)^\gamma|\xi|^{2\gamma} f\|_2
\leq C_\gamma \|\widetilde{L}_{\xi}^{\gamma} f\|_2.
\end{equation}
\end{proposition}
\begin{proof}
Set $\widetilde{L}=-\Delta_{d_1}+\sum_{{{j}}=1}^{d_1}|x'_{{j}}|$ and
next define operator $L_{x'_i}$ by the following formula
$$
L_{x'_i}=-\partial_ {x'_i}^2+|x'_i|.
$$
By Proposition 3.4 of \cite{G}
$$
\||x'_i|^k f\|_2\leq C'_k \|L_{x'_i}^{k} f\|_2
$$
for all positive natural numbers $k\in \N$.
Hence
$$
\|\, (\sum_{{{j}}=1}^{d_1}|x'_{{j}}|)^k f\|_2^2
 \leq C \sum_{{{j}}=1}^{d_1}\| |x'_{{j}}|^k f\|_2^2
 \leq C_k  \sum_{{{j}}=1}^{d_1}
\|L_{x'_j}^{k} f\|_2^2.
$$
Note that all $L_{x'_i}$ are non-negative self-adjoint operators and
commute strongly, that is, their resolvent  commute. Therefore
for all $\ell_i\in \mathbb{Z}_+$, operators
$\prod_{i=1}^nL_{x'_i}^{\ell_i}$ are self-adjoint and non-negative.
Hence
$$
\sum_{{{j}}=1}^{d_1}L_{x'_j}^{2k}\leq(\sum_{{{j}}=1}^{d_1}L_{x'_j})^{2k}
$$
for all $k\in \N$ and 
\begin{eqnarray*}
\|\, (\sum_{{{j}}=1}^{d_1}|x'_{{j}}|)^k f\|^2_2 &\leq& C_k
\sum_{{{j}}=1}^{d_1}\|L_{x'_j}^{k}
f\|^2_2=C_k\langle\sum_{{{j}}=1}^{d_1}L_{x'_j}^{2k}f,f\rangle\\
&\leq& C_k\langle(\sum_{{{j}}=1}^{d_1}L_{x'_j})^{2k}f,f\rangle =C_k
\|\widetilde{L}^{k} f\|^2_2.
\end{eqnarray*}
Next, for a function $f \in C^\infty_c(\R^{d_1})$ we define function $\delta_tf$  by the formula $\delta_tf(x)=f(tx)$.
Note that if $t=|\xi|^{-2/3}$ then
$$
\widetilde{L}^k_{\xi}=\left(-\Delta_{d_1}+\left(\sum_{{{j}}=1}^{d_1}|x'_{{j}}|\right)t^{-3}\right)^{k}=
t^{-2k}\delta_{t^{-1}}\widetilde{L}^k\delta_t.
$$
Hence
\begin{eqnarray}
\|\widetilde{L}_{\xi}^k f\|_2=
\|t^{-2k}\delta_{t^{-1}}\widetilde{L}^k\delta_tf\|_2
&=&t^{-2k}t^{d_1/2}\|\widetilde{L}^k\delta_tf\|_2\nonumber\\
&\geq& C''_k t^{-2k}t^{d_1/2}\|\,
(\sum_{{{j}}=1}^{d_1}|x'_{{j}}|)^k \delta_tf\|_2\nonumber\\
&=&C_k''  |\xi|^{2k}\|\, (\sum_{{{j}}=1}^{d_1}|x'_{{j}}|)^k
f\|_2\nonumber.
\end{eqnarray}
This proves Proposition~\ref{prop2.2} for all $\gamma=k \in \N$.
Now in virtue of L\"owner-Heinz inequality (see, e.g., \cite[Section I.5]{Co})
we can extend these estimates to all $\gamma\in
[0,\infty)$.
\end{proof}

\section{Crucial estimates}\label{sec3}

To be able to obtain a required description of spectral decomposition of the operators $\widetilde{L}_{\xi}$
 we need  the following  properties of spectral
decomposition of operator $A=-\frac{d^2}{d x^2}+|x|$ acting on
$L^2(\R)$ and which are essentially based on results from  \cite{G}.

\begin{proposition}\label{2prop3.1}
Let $\lambda_n$ and $h_n$ be the $n$-th eigenvalue and normalized
eigenfunction of the operator $A=-\frac{d^2}{d x^2}+|x|$.   Then its
spectral decomposition satisfies following properties:

\smallskip
 {\rm (i)} The operator $A$ has only a pointwise spectrum
and its eigenvalues  belong to $(1,\infty)$. In particular  the
first eigenvalue is larger than $1$.

\smallskip
{\rm (ii)} Every eigenvalue of $A$ is simple and the only point of
accumulation of the eigenvalue sequence is $\infty$. Thus
$\{h_n\}_{n\in \N}$ is a complete orthonormal system of $L^2(\R)$.

\smallskip
{\rm (iii)} The eigenvalues $\lambda_n$ satisfy the following
estimates:
\begin{eqnarray}\label {eeqq2.3}
C_1\left(\frac{3\pi}{4}n\right)^{2/3} \leq \lambda_n\leq C_2
\left(\frac{3\pi}{4}n\right)^{2/3},
\end{eqnarray}
\begin{eqnarray}\label {eeqq2.4}
\frac{\pi}{2}\lambda_{n+1}^{-1/2}\leq \lambda_{n+1}-\lambda_n\leq
\frac{\pi}{2}\lambda_{n}^{-1/2},
\end{eqnarray}
where $C_2\geq C_1>0$ are constants.

\smallskip
{\rm (iv)} For the eigenfunction $h_n$ corresponding to the
eigenvalue $\lambda_n$,
\begin{equation}\label{eq16}
h_n(u) \leq \begin{cases}
C\lambda_n^{-\frac{1}{4}}(1 + \big||u|-\lambda_n\big|)^{-\frac{1}{4}}, &\text{ $u \in \R$,}\\
C\exp(-c|u|^{\frac{3}{2}}), &\text{ $u \geq 2\lambda_n$.}
\end{cases}
\end{equation}
\end{proposition}

\begin{proof}
 (i), (ii) and (iii) are just reformulation of Proposition
2.1, Corollary 2.2, Facts 2.3, 2.7 and 2.8 of \cite{G}. (iv) is an
easy consequence of Theorem 2.6 of \cite{G} and estimates for Airy
function (see for example \cite{Ho}, pp. 213-215).
\\
\end{proof}

Now we are able to describe  spectral resolutions  of Grushin operator $L=L_1$ and operators $\widetilde{L}_{\xi}$
defined in Section~\ref{sec2}. It is interesting to compare it with spectral decomposition of the operator $L_2$ 
obtained in \cite{MaS12}. Spectral decompositions of $L_1$ and $L_2$ are significantly different  even though 
they share many common features. 
We also investigate integral kernels of spectral multipliers of $L$ and  $\widetilde{L}_{\xi}$. 
For $T= F(L)$ or $T=F(\widetilde{L}_{\xi})$, by  $K_{T} $  we denote the integral kernel of the operator
$T$, defined by the identity
\begin{equation*}
T f(x) = \int_{X} K_{T}(x,y) \, f(y) \,dy
\end{equation*}
where $X={\R^{d_1}\times\R^{d_2}}$ for $L$ and  $X=\R^{d_1}$ for $\widetilde{L}_{\xi}$. 

In terms of the eigenvalues and eigenfunctions of the operator
$A=-\frac{d^2}{d x^2}+|x|$, one can obtain explicit formula for the
integral kernel of the operator $F(L)$, compare also \cite[Proposition 5]{MaS12}.
Let $\lambda_n$ and $h_n$ be the $n$-th eigenvalue and eigenfunction
of the operator $-\frac{d^2}{d x^2}+|x|$ on $L^2(\R)$.  We know
that $\{h_n\}_{n\in \N}$ is a complete orthonormal system of
$L^2(\R)$. For all positive integers $d_1$, all  ${\bf n} \in
\N^{d_1}$  and  all $\xi \in
\R^{d_2}$, we define function $\tilde h_{d_1,{\bf n}} \colon \R^{d_1} \to \R$ by the formula
\[\tilde h_{d_1,{\bf n}}(x',\xi) = |\xi|^{d_1/3} h_{n_1}(|\xi|^{2/3} x'_1) \cdots h_{n_{d_1}}(|\xi|^{2/3} x'_{d_1}).\]
We are now able to describe the kernel $K_{F(L)}$.

\begin{proposition}\label{prop3.2}
For all bounded compactly supported Borel functions $F : \R \to \C$
\begin{eqnarray*}
K_{F(L)}(x,y) &=&(2\pi)^{-d_2} \int_{\R^{d_2}} K_{F(\widetilde{L}_\xi)}(x',y') \,
e^{i\xi \cdot (x''-y'')} \,d\xi\nonumber\\
&=& (2\pi)^{-d_2} \int_{\R^{d_2}} \sum_{{\bf n} \in \N^{d_1}}
F(\sum_{i=1}^{d_1}|\xi|^{\frac{4}{3}}\lambda_{n_i}) \, \tilde h_{d_1,{\bf n}}(y',\xi) \, 
\tilde h_{d_1,{\bf n}}(x',\xi) \, e^{i\xi \cdot (x''-y'')}
\,d\xi
\end{eqnarray*}
for almost   all $x=(x',x'') ,y=(y',y'') \in \R^{d_1}\times\R^{d_2}$.
\end{proposition}

\begin{proof}
We noticed in Section \ref{sec2} that $\mathcal{F} L\phi (x',\xi) = \widetilde{L}_{\xi} \, \mathcal{F}
\phi(x',\xi)$ where $\mathcal{F}$ is  the partial Fourier transform in variables $x''$.
Next note that for all $\xi \neq 0$
\[\widetilde{L}_{\xi} \, \tilde h_{d_1,{\bf n}}(x',\xi) =\left(\sum_{j=1}^{d_1} |\xi|^{\frac{4}{3}}\lambda_{n_j}\right) \,
\tilde h_{d_1,{\bf n}}(x',\xi).\] Moreover by 
Propostition~\ref{2prop3.1} (ii),  the set $\{\tilde h_{d_1,{\bf n}}(x',\xi)\}_{{\bf n}\in \N^{d_1}}$ is a complete orthonormal system
of $L^2(\R^{d_1})$. Hence if $\mathcal{G} : L^2(\R^{d_1} \times
\R^{d_2}) \to L^2(\N^{d_1} \times \R^{d_2})$ is the isometry defined
by
\[\mathcal{G} \psi(n,\xi) = \int_{\R^{d_1}} \psi(x',\xi) \, \tilde h_{d_1,{\bf n}}(x',\xi) \,dx',\]
then
\[
\mathcal{G} \mathcal{F} L \phi({\bf n},\xi) = \sum_{j=1}^{d_1}
|\xi|^{\frac{4}{3}}\lambda_{n_j} \, \mathcal{G} \mathcal{F}
\phi({\bf n},\xi)
\]
and
\begin{eqnarray}\label{2eq3.1}
\mathcal{G} \mathcal{F} \, F(L) \, \phi({\bf n},\xi) =
F(\sum_{j=1}^{d_1} |\xi|^{\frac{4}{3}}\lambda_{n_j}) \, \mathcal{G}
\mathcal{F} \phi({\bf n},\xi).
\end{eqnarray}
However the inverse of $\mathcal{G}$ is given by 
$$
\mathcal{G}^{-1}\varphi(x',\xi)=\sum_{{\bf n}\in
\N^{d_1}}\varphi({\bf n},\xi)\tilde h_{d_1,{\bf n}}(x',\xi).
$$
and inverse of $\mathcal{F}$ can be expressed in terms of partial inverse Fourier 
transform in $x''$. 
Applying  $\mathcal{G}^{-1}$ and $\mathcal{F}^{-1}$ to both sides of equality \eqref{2eq3.1} shows Proposition \ref{prop3.2}.
\end{proof}

Next, for all positive integers $d_1$ and all ${\bf n} \in
\N^{d_1}$  we define function  $H_{d_1,{\bf n}} \colon  \R^{d_1} \to \R$ by the formula
\[H_{d_1,{\bf n}}(x') =  h^2_{n_1}(x'_1)
\cdots h^2_{n_{d_1}}(x'_{d_1}).\]
As a simple consequence of Proposition~\ref{prop3.2} we obtain following estimates. 

\begin{proposition}\label{prop3.3}
For all $\gamma \geq 0$ and for every   compactly supported bounded Borel function $F : \R \to \C$,
\begin{equation*}
\Big\|\,\big(\sum_{i=1}^{d_1}|x'_i|\big)^\gamma
K_{F(L)}(\cdot,y)\Big\|_2^2 \leq C_\gamma \int_0^\infty
|F(\theta)|^2 \sum_{{\bf n} \in \N^{d_1}}
\frac{\theta^{Q/2-\gamma}}{N_{\bf n}^{Q/2-3\gamma}} \, H_{d_1,{\bf
n}}\left(\frac{\theta^{1/2} y'}{N_{\bf n}^{1/2}}\right)
\,\frac{d\theta}{\theta}
\end{equation*}
for almost all $y=(y',y'') \in \R^{d_1}\times \R^{d_2}$ where $N_{\bf
n}=\sum_{i=1}^{d_1}\lambda_{n_i}$ and $\lambda_{n_i}$ is the
eigenvalue corresponding to eigenfunction $h_{n_i}$.
\end{proposition}
\begin{proof} By Propositions~\ref{prop2.2} and \ref{prop3.2}
\begin{eqnarray}\label{eq13}
\Big\|\,\big(\sum_{i=1}^{d_1}|x'_i|\big)^\gamma
K_{F(L)}(\cdot,y)\Big\|_2^2
&=&\int_{\R^{d_2}}\Big\|\big(\sum_{i=1}^{d_1}|x'_i|\big)^\gamma
K_{F(\widetilde{L}_\xi)}(x',y')\Big\|^2_{L^{2}(\R^{d_1})}d\xi\nonumber\\
&\leq&\int_{\R^{d_2}}|\xi|^{-4\gamma}\Big\|\widetilde{L}_\xi^\gamma
K_{F(\widetilde{L}_\xi)}(x',y')\Big\|^2_{L^{2}(\R^{d_1})}d\xi.
\end{eqnarray}
Next note that for all $\gamma \geq 0$ and  $y' \in \R^{d_1}$
\[\widetilde{L}_\xi^{\gamma}  \left( K_{F(\widetilde{L}_\xi)}(\cdot,y') \right)
 = K_{\widetilde{L}_\xi^{\gamma} F(\widetilde{L}_\xi)}(\cdot, y')\]
Hence 
\begin{eqnarray}\label {2eqq19}
 \|\widetilde{L}_\xi^\gamma
K_{F(\widetilde{L}_\xi)}(x',y')\|^2_{L^{2}(\R^{d_1})}
&\leq& \|K_{\widetilde{L}_\xi^\gamma F(\widetilde{L}_\xi)}(x',y')\|^2_{L^{2}(\R^{d_1})}\nonumber\\
&\leq&  \sum_{{\bf n}\in
\N^{d_1}}\left|\big(\sum_{i=1}^{d_1}|\xi|^{\frac{4}{3}}\lambda_{n_i}\big)^\gamma
F(\sum_{i=1}^{d_1}|\xi|^{\frac{4}{3}}\lambda_{n_i})\right|^2|{\tilde h_{d_1,{\bf n}}(y',\xi)}|^2\nonumber\\
&\leq& C|\xi|^{\frac{2d_1}{3}+\frac{8\gamma}{3}}\sum_{{\bf n}\in
\N^{d_1}}N_{\bf n}^{2\gamma}|F(|\xi|^{\frac{4}{3}}N_{\bf
n})|^2H_{d_1,{\bf n}}(|\xi|^{\frac{2}{3}}y').
\end{eqnarray}
Now substituting   \eqref{2eqq19} to 
 \eqref{eq13} and simple change of variables proves Proposition \ref{prop3.3}

\end{proof}

The following lemma is a version of Lemma 9 of \cite{MaS12}. However the proof is more 
complex and
requires a new approach especially when $d_1 \ge 2$. It is the most essential part of the proof of
our main spectral multiplier results. 

\begin{lemma}\label{lem3.4}
For all $\varepsilon  >0$ there exists a constant $C>0$ which does not
depend on   $x' \in \R^{d_1}$   such that
\begin{equation}\label{eqairy}
\sum_{{\bf n} \in \N^{d_1}}
\frac{\max\{1,|x'|\}^{2\varepsilon }}{N_{\bf n}^{d_1/2+3\varepsilon }} \, H_{d_1,{\bf n}}
\left( \frac{x'}{N_{\bf n}^{1/2}}\right) < C < \infty
\end{equation}
 where
$N_{\bf n}=\sum_{i=1}^{d_1}\lambda_{n_i}$ and $\lambda_{n_i}$ is the
eigenvalue corresponding to eigenfunction $h_{n_i}$.
\end{lemma}
\begin{proof}
We split the sum into two parts,
\begin{eqnarray}
&&\sum_{{\bf n} \in \N^{d_1}}
\frac{\max\{1,|x'|\}^{2\varepsilon }}{N_{\bf n}^{d_1/2+3\varepsilon }} \,
H_{d_1,{\bf n}} \left( \frac{x'}{N_{\bf
n}^{1/2}}\right)\nonumber\\
&&\leq \left(\sum_{N_{{\bf n}}^{3/2}\leq |x'|/(2d_1)}+\sum_{N_{{\bf
n}}^{3/2}> |x'|/(2d_1)}\right)
\frac{\max\{1,|x'|\}^{2\varepsilon }}{N_{\bf n}^{d_1/2+3\varepsilon }} \,
H_{d_1,{\bf n}} \left( \frac{x'}{N_{\bf n}^{1/2}}\right).
\end{eqnarray}
\noindent {\it Part 1:} $N_{{\bf n}}^{3/2}\leq |x'|/(2d_1)$. By
Proposition~\ref{2prop3.1}  $\lambda_{n_i} \ge 1$   so  $N_{{\bf n}}>1$.
Hence this part
is empty unless $|x'|>1$. Note that
$$
 \frac{|x'|_\infty}{N_{\bf n}^{1/2}}\geq  \frac{|x'|}{d_1N_{\bf n}^{1/2}}
 \geq 2N_{{\bf n}}
$$
where $|x'|_{\infty}=\max\{x'_1,\cdots,x'_{d_1}\}$. By \eqref{eq16} for every
natural number $N\leq |x'|/(2d_1)$
$$
\sum_{(N-1)^{2/3}<N_{\bf n}\leq N^{2/3}}H_{d_1,{\bf n}} \left(
\frac{x'}{N_{\bf n}^{1/2}}\right) \leq
C\exp(-c|x'|_\infty^\frac{3}{2}/N^{\frac{1}{2}})\leq
C\exp(-c|x'|^\frac{3}{2}/N^{\frac{1}{2}}).
$$
Thus 
\begin{eqnarray}
&&\sum_{N_{{\bf n}}^{3/2}\leq |x'|/(2d_1)}
\frac{\max\{1,|x'|\}^{2\varepsilon }}{N_{\bf n}^{d_1/2+3\varepsilon }} \,
H_{d_1,{\bf n}} \left( \frac{x'}{N_{\bf
n}^{1/2}}\right)\nonumber\\
&&\leq \sum_{N\leq |x'|/(2d_1)}\sum_{(N-1)^{2/3}<N_{\bf n}\leq
N^{2/3}} \frac{\max\{1,|x'|\}^{2\varepsilon }}{N_{\bf
n}^{d_1/2+3\varepsilon }} H_{d_1,{\bf n}} \left( \frac{x'}{N_{\bf
n}^{1/2}}\right)\nonumber\\
&&\leq C \sum_{N \leq |x'|/2} {|x'|}^{2\varepsilon }
N^{-d_1/3-2\varepsilon }
\exp(-c|x'|^\frac{3}{2}/N^{\frac{1}{2}})\nonumber\\
&&\leq C \sum_{N \in \N} \sup_{t \geq 2N} t^{4\varepsilon /3} \exp(-ct)
\leq C.
\end{eqnarray}

\medskip

\noindent {\it Part 2:} $N_{{\bf n}}^{3/2}> |x'|/(2d_1)$.
Again  by  \eqref{eq16}
\begin{eqnarray*}
H_{d_1,{\bf n}} \left( \frac{x'}{N_{\bf n}^{1/2}}\right)
&=&\prod_{i=1}^{d_1}h^2_{n_i}\left(
\frac{x'_i}{N_{\bf n}^{1/2}}\right)\\
&\leq&C \prod_{i=1}^{d_1}\lambda_{n_i}^{-\frac{1}{2}}
\left(1+\left|\frac{|x'_i|}{N_{\bf
n}^{1/2}}-\lambda_{n_i}\right|\right)^{-\frac{1}{2}}.
\end{eqnarray*}
Hence 
\begin{eqnarray}\label {eeqq20}
&&\sum_{N_{{\bf n}}^{3/2}> |x'|/(2d_1)}
\frac{\max\{1,|x'|\}^{2\varepsilon }}{N_{\bf n}^{d_1/2+3\varepsilon }}
H_{d_1,{\bf n}} \left( \frac{x'}{N_{\bf n}^{1/2}}\right)\nonumber\\
&&\leq C \sum_{N_{{\bf n}}^{3/2}> |x'|/(2d_1)}
\frac{\max\{1,|x'|\}^{2\varepsilon }}{N_{\bf
n}^{d_1/2+3\varepsilon }}\prod_{i=1}^{d_1}\lambda_{n_i}^{-\frac{1}{2}}
\left(1+\left|\frac{|x'_i|}{N_{\bf
n}^{1/2}}-\lambda_{n_i}\right|\right)^{-\frac{1}{2}}.
\end{eqnarray}
Next, define function $g \colon [1,\infty)^{d_1} \to \R_+$ by the
formula
$$
g(\mu)=g(\mu_{1}, \ldots,
\mu_{{d_1}})=\frac{\max\{1,|x'|\}^{2\varepsilon }}{N_{\bf
\mu}^{d_1/2+3\varepsilon }}\prod_{i=1}^{d_1}\mu_{i}^{-\frac{1}{3}}
\left(1+\left|\frac{|x'_i|}{N_{\bf
\mu}^{1/2}}-\mu_{i}^{2/3}\right|\right)^{-\frac{1}{2}}
$$
where $N_{\bf \mu}=\sum_{i=1}^{d_1} \mu_{i}^{2/3}$.  Note that
$g(\mu_{1}, \ldots, \mu_{{d_1}})>0$ and there exists a constant
$C>0$ such that
$$
\left| \nabla g(\mu_{1}, \ldots, \mu_{{d_1}})\right| \le C
g(\mu_{1}, \ldots, \mu_{{d_1}})
$$
when  $\mu= (\mu_{1}, \ldots \mu_{{d_1}})\in [1,\infty)^{d_1}$ and  $N_{\bf \mu}=\sum_{i=1}^{d_1} \mu_{i}^{2/3}\ge
(|x'|/(2d_1))^{2/3}$. By the above estimate for the gradient of $g$
$$
e^{-C|\mu-\bar \mu|}\le \left|\frac{g(\mu)}{g(\bar\mu)}\right| \le e^{C|\mu-\bar\mu|}
$$
for all $\mu, \bar \mu$ in the region described above. 
Hence
\begin{eqnarray}\label{2eq28}
g(\mu_{1}, \ldots, \mu_{{d_1}})\le C
\int_{\prod_{i=1}^{d_1}[\mu_{i},\mu_{i}+1 ]} g(\xi_{1}, \ldots,
\xi_{{d_1}}) d\xi_{1} \ldots d\xi_{{d_1}}.
\end{eqnarray}
Set $\mu_{n_i}=\lambda_{n_i}^{3/2}$.  By \eqref{eeqq20},
\begin{eqnarray}\label{2eqq29}
\sum_{N_{{\bf n}}^{3/2}> |x'|/(2d_1)}
\frac{\max\{1,|x'|\}^{2\varepsilon }}{N_{\bf n}^{d_1/2+3\varepsilon }}
H_{d_1,{\bf n}} \left( \frac{x'}{N_{\bf n}^{1/2}}\right)\leq
\sum_{N_{{\bf n}}^{3/2}> |x'|/(2d_1)} g(\mu_{n_1}, \ldots,
\mu_{n_{d_1}}).
\end{eqnarray}
However , by (\ref{eeqq2.4}) and mean value theorem for
each $1\leq i\leq d_1$,
\begin{eqnarray}\label{2eq29}
\mu_{n_i}-\mu_{n_i-1}=\lambda_{n_i}^{3/2}-\lambda_{n_i-1}^{3/2}&\geq& \frac{3\pi}{4}\lambda_{n_i}^{-1/2}\lambda_{n_i-1}^{1/2}\nonumber\\
&\geq& \frac{3\pi}{4}\left(\frac{\lambda_{n_i-1}^{3/2}}{\frac{\pi}{2}+\lambda_{n_i-1}^{3/2}}\right)^{1/2}\nonumber\\
&\geq& \frac{3\pi}{8} > 1
\end{eqnarray}
which means that for all ${\bf n}\in \N^{d_1}$, cubes 
$\prod_{i=1}^{d_1}[\mu_{n_i},\mu_{n_i}+1]$ are mutually  disjoint. Note again 
 that by Proposition~\ref{2prop3.1}  $\lambda_{n_i} \ge 1$ so $N_{{\bf
n}}>1$. Hence by \eqref{2eq28}, \eqref{2eqq29} and
\eqref{2eq29} 
\begin{eqnarray*}
&&\sum_{N_{{\bf n}}^{3/2}> |x'|/(2d_1)}
\frac{\max\{1,|x'|\}^{2\varepsilon }}{N_{\bf n}^{d_1/2+3\varepsilon }}
H_{d_1,{\bf n}} \left( \frac{x'}{N_{\bf n}^{1/2}}\right)\\
&& \le C \int_{N_{{\bf \mu}}> \max\{(|x'|/(2d_1))^{2/3},1\}}
g(\mu_{1}, \ldots, \mu_{{d_1}})d\mu_{1} \ldots d\mu_{{d_1}}.
\end{eqnarray*}
Using the changes of variables  $\mu_i=\xi_i^{3/2}$ we get 
\begin{eqnarray}\label{eeqqq26}
&&\sum_{N_{{\bf n}}^{3/2}> |x'|/(2d_1)}
\frac{\max\{1,|x'|\}^{2\varepsilon }}{N_{\bf n}^{d_1/2+3\varepsilon }}
H_{d_1,{\bf n}} \left( \frac{x'}{N_{\bf n}^{1/2}}\right)\nonumber\\
&&\leq C \int_{N_{\bf \xi}> \max\{(|x'|/(2d_1))^{2/3},1\}}
\left[\frac{\max\{1,|x'|\}^{2\varepsilon }}{N_{\bf
\xi}^{d_1/2+3\varepsilon }}N_{\bf
\xi}^{\frac{d_1}{4}}\prod_{i=1}^{d_1}\xi_{i}^{-\frac{1}{2}}
\big||x'_i|-\xi_{i}N_{\bf
\xi}^{\frac{1}{2}}\big|^{-\frac{1}{2}}\right]d\xi_{1}^{\frac{3}{2}}
\ldots d\xi_{d_1}^{\frac{3}{2}}\nonumber\\
&&\leq C \int_{N_{\bf \xi}> \max\{(|x'|/(2d_1))^{2/3},1\}}
\left[\frac{\max\{1,|x'|\}^{2\varepsilon }}{N_{\bf
\xi}^{d_1/4+3\varepsilon }}\prod_{i=1}^{d_1} \big||x'_i|-\xi_{i}N_{\bf
\xi}^{\frac{1}{2}}\big|^{-\frac{1}{2}}\right]d\xi_{1} \ldots
d\xi_{d_1}=I
\end{eqnarray}
where $N_{\xi}=\sum_{i=1}^{d_1}\xi_i$. To estimate this integral we
use the following decomposition 
\begin{eqnarray*}
&&\{{\bf \xi}\, \colon \,N_{\bf \xi}\geq \max\{(|x'|/(2d_1))^{2/3},1\}\}\\
&&\quad\quad\quad=\bigcup_{j=1}^{d_1}E_j=\bigcup_{j=1}^{d_1}\{{\bf
\xi}\, \colon \, N_{\bf \xi}\geq \max\{(|x'|/(2d_1))^{2/3},1\}, N_\xi/d_1\leq
\xi_j\leq N_\xi\}.
\end{eqnarray*}
Now   on each of  set $E_j$ we introduce new coordinates
$$
\nu_1=\xi_1, \ldots, \nu_{j-1}=\xi_{j-1}, \nu_{j}=N_\xi,
\nu_{j+1}=\xi_{j+1}, \ldots,  \nu_{d_1}=\xi_{d_1}.
$$
Then
\begin{eqnarray}\label{eeqq21}
&&I \leq C\sum_{j=1}^{d_1}
\int_{\max\{(|x'|/(2d_1))^{2/3},1\}}^{\infty}\frac{\max\{1,|x'|\}^{2\varepsilon }}{\nu_{j}^{d_1/4+3\varepsilon }} \nonumber \\
&&\quad\quad\times \int_{S_j} \prod_{i\neq j}
\big||x'_i|-\nu_{i}\nu_j^{\frac{1}{2}}\big|^{-\frac{1}{2}}\big||x'_j|-\bar{\nu_{j}}\nu_j^{\frac{1}{2}}\big|^{-\frac{1}{2}}
d\nu_1 \ldots d\nu_{d_1}
\end{eqnarray}
where  $\bar{\nu_j}=\nu_j-\sum_{i\neq j}\nu_i$ and $S_j=\{\nu\, \colon \, \nu_j/d_1 \leq \bar{\nu_j}\leq \nu_j,\, 0\leq
\nu_i\leq \nu_j, \forall i\neq j \}$.

\smallskip
 Next we split the integral into two parts:
$\nu_j>\max\{(2d_1|x'|)^{2/3},1\}$ and $(|x'|/(2d_1))^{2/3}\leq
\nu_j\leq (2d_1|x'|)^{2/3}$.
 Note that if
$\nu_j\geq (2d_1|x'|)^{2/3}$ and $\nu_j/d_1\leq \bar{\nu_j}\leq
\nu_j$ then
$$
\big||x'_j|-\bar{\nu_{j}}\nu_j^{\frac{1}{2}}\big|^{-\frac{1}{2}}
\leq C\nu_j^{-3/4}.
$$
Note also that there exists a constant $C$ such that for all $A, N >0$
$$
\int_0^N \left|  A- x   \right|^{-1/2} dx \le CN^{1/2}.
$$
Hence for $\nu_j>\max\{(2d_1|x'|)^{2/3},1\}$,
\begin{eqnarray*}
&&\int_{S_j} \prod_{i\neq j}
\big||x'_i|-\nu_{i}\nu_j^{\frac{1}{2}}\big|^{-\frac{1}{2}}\big||x'_j|-\bar{\nu_{j}}\nu_j^{\frac{1}{2}}\big|^{-\frac{1}{2}}
d\nu_1 \ldots d\nu_{j-1}d\nu_{j+1}\ldots d\nu_{d_1}\\
&&\leq C\nu_j^{-3/4}\prod_{i\neq
j}\int_{0}^{\nu_j}\big||x'_i|-\nu_{i}\nu_j^{\frac{1}{2}}\big|^{-\frac{1}{2}}d\nu_i\\
&&\leq C\nu_j^{-3/4}\nu_j^{\frac{d_1-1}{4}}\leq C\nu_j^{d_1/4-1}
\end{eqnarray*}
and
\begin{eqnarray}\label {eqq21}
&&\int_{\max\{(2d_1|x'|)^{2/3},1\}}^{\infty}\frac{\max\{1,|x'|\}^{2\varepsilon }}{\nu_{j}^{d_1/4+3\varepsilon }}
\int_{S_j} \prod_{i\neq j}
\big||x'_i|-\nu_{i}\nu_j^{\frac{1}{2}}\big|^{-\frac{1}{2}}\big||x'_j|-\bar{\nu_{j}}\nu_j^{\frac{1}{2}}\big|^{-\frac{1}{2}}
d\nu_1 \ldots d\nu_{d_1}\nonumber\\
&&\leq
C\int_{\max\{(2d_1|x'|)^{2/3},1\}}^{\infty}\frac{\max\{1,|x'|\}^{2\varepsilon }}{\nu_j^{d_1/4+3\varepsilon }}\nu_j^{d_1/4-1}
d\nu_j\nonumber\\
&&\leq
C\int_{\max\{(2d_1|x'|)^{2/3},1\}}^{\infty}\frac{\max\{1,|x'|\}^{2\varepsilon }}{\nu_j^{1+3\varepsilon }}d\nu_j
\leq C.
\end{eqnarray}
\bigskip

If we assume now that  $(|x'|/(2d_1))^{2/3}\leq \nu_j \leq
(2d_1|x'|)^{2/3}$ then by the change of variables $\nu_i
\nu_j^{\frac{1}{2}}=u_i$ one gets 
\begin{eqnarray}
&&\int_{S_j} \prod_{i\neq j}
\big||x'_i|-\nu_{i}\nu_j^{\frac{1}{2}}\big|^{-\frac{1}{2}}\big||x'_j|-\bar{\nu_{j}}\nu_j^{\frac{1}{2}}\big|^{-\frac{1}{2}}
d\nu_1 \ldots d\nu_{j-1}d\nu_{j+1}\ldots d\nu_{d_1}\nonumber\\
&&\leq C\nu_j^{\frac{1-d_1}{2}}
\int_{[0,\nu^{3/2}_j]^{d_1-1}}\big||x'_i|-u_i\big|^{-\frac{1}{2}}\big||x'_j|+\sum_{i\neq
j}u_i-\nu_j^{\frac{3}{2}}\big|^{-\frac{1}{2}}d{\bf u}\nonumber
\end{eqnarray}
where $d{\bf u}=du_1\cdots du_{j-1}du_{j+1}\cdots du_{d_1}$. Hence,
 \begin{eqnarray*}
 &&\int_{(|x'|/(2d_1))^{2/3}}^{(2d_1|x'|)^{2/3}}\frac{\max\{1,|x'|\}^{2\varepsilon }}{\nu_{j}^{d_1/4+3\varepsilon }}
\int_{S_j} \prod_{i\neq j}
\big||x'_i|-\nu_{i}\nu_j^{\frac{1}{2}}\big|^{-\frac{1}{2}}\big||x'_j|-\bar{\nu_{j}}\nu_j^{\frac{1}{2}}\big|^{-\frac{1}{2}}
d\nu_1 \ldots d\nu_{d_1}\nonumber\\
&&\leq
C\int_{(|x'|/(2d_1))^{2/3}}^{(2d_1|x'|)^{2/3}}\nu_j^{\frac{2-3d_1}{4}}
\int_{[0,\nu^{3/2}_j]^{d_1-1}}\big||x'_i|-u_i\big|^{-\frac{1}{2}}\big||x'_j|+\sum_{i\neq
j}u_i-\nu_j^{\frac{3}{2}}\big|^{-\frac{1}{2}}d{\bf u}d\nu_j\nonumber\\
&&\leq
C|x'|^{\frac{2-3d_1}{6}}\int_{[0,2d_1|x'|]^{d_1-1}}\prod_{i\neq
j}\big||x'_i|-u_i\big|^{-\frac{1}{2}}\nonumber\\
&&\quad\quad\quad\quad\quad\quad\times\int_{(|x'|/(2d_1))^{2/3}}^{(2d_1|x'|)^{2/3}}
\big||x'_j|+\sum_{i\neq
j}u_i-\nu_j^{\frac{3}{2}}\big|^{-\frac{1}{2}}d\nu_j d{\bf
u}\nonumber\\
&&\leq C|x'|^{\frac{2-3d_1}{6}}|x'|^{1/6}\prod_{i\neq
j}\int_0^{2d_1|x'|}\big||x'_i|-u_i\big|^{-\frac{1}{2}}du_i\nonumber\\
&&\leq C|x'|^{\frac{2-3d_1}{6}}|x'|^{1/6}|x'|^{(d_1-1)/2}\nonumber\\
&&\leq C.
\end{eqnarray*}
Now \eqref{eeqqq26}, \eqref{eeqq21}, \eqref{eqq21} and
the above estimates  yield 
\begin{eqnarray*}
\sum_{N_{{\bf n}}^{3/2}> |x'|/(2d_1)}
\frac{\max\{1,|x'|\}^{2\varepsilon }}{N_{\bf n}^{d_1/2+3\varepsilon }}
H_{d_1,{\bf n}} \left( \frac{x'}{N_{\bf n}^{1/2}}\right)\leq C.
\end{eqnarray*}

\end{proof}

\bigskip 

Next, for all  $R > 0$ we define the  weight function $w_R \colon  (\R^{d_1}\times \R^{d_2})^2 \to \R_+$ by
the formula 
\begin{equation*}
w_R(x,y) = \min\{R,|y'|^{-1}\} |x'|.
\end{equation*}
The estimates obtained in this section can be summarised in the following proposition.

\begin{proposition}\label{prop3.5}
For all $\gamma \in [ 0,d_2/4 )$ and all bounded compactly supported
Borel functions $F : \R \to \C$,
\[
\left\|\, \Big|\sum_{i=1}^{d_1}|x'_i|\Big|^\gamma K_{F(L)}(\cdot,y)\right\|_2^2 \leq
C_\gamma \int_0^\infty |F(\lambda)|^2 \, \lambda^{(d_1+d_2)/2} \,
\min\{\lambda^{d_2/4-\gamma},|y'|^{2\gamma-d_2/2}\}
\,\frac{d\lambda}{\lambda}
\]
for almost all $y=(y', y'') \in \R^{d_1}\times \R^{d_2}$. In particular, for
all $R
> 0$, if $\supp F \subseteq \left[ R^2,4R^2
\right]$, then
\[
\esssup_{y \in \R^{d_1}\times\R^{d_2}}  |B(y,R^{-1})|^{1/2} \, \|
w_R(\cdot,y)^\gamma K_{F(L)}(\cdot,y)\|_2 \leq C_{\gamma}
\|\delta_{R^2} F\|_{L^2},
\]
where the constant $C_{\gamma}$ does not depend on $R$.
\end{proposition}
\begin{proof}
We obtain the first inequality by Proposition \ref{prop3.3}
and Lemma~\ref{lem3.4} with $\varepsilon  = d_2/4-\gamma$. Next if we assume that 
$\supp F \subseteq \left[ R^2,4R^2 \right]$, then in virtue of the definition of the weight 
$ w_R$ and  estimate \eqref{eqVr}, 
the first inequality implies the second one.
\end{proof}

\section{The multiplier theorems}\label{sec4}

In the following section we show that Theorems \ref{thm1m} and \ref{thm2r} are straightforward consequence of
Proposition~\ref{prop3.5}. The argument is essential  the same as in Section 5 of \cite{MaS12} with an obvious 
adjustment of exponents in some calculations and we quote it here for sake of completeness. An alternative proof
based on the wave equation technique can be obtain by a simple modification of the proof of 
\cite[Lemma 3.4]{CS}.

\begin{proposition}\label{stw}
For all $R > 0$, $\alpha \geq 0$, $\beta > \alpha$, and for all functions $F : \R \to \C$ such that $\supp F \subseteq [-4R^2,4R^2]$,
\begin{equation}\label{eqSW}
\esssup_{y \in \R^{d_1}\times\R^{d_2}} |B(y,R^{-1})|^{1/2} \|(1+ R\rho(\cdot,y))^\alpha K_F(\cdot,y)\|_2 \leq C_{\alpha,\beta} \|\delta_{R^2} F\|_{W_\infty^\beta},
\end{equation}
where the constant $C_{\alpha,\beta}$ does not depend on $R$. If in addition  $\beta > \alpha+Q/2$, then
\begin{equation}\label{eqstw}
\esssup_{y \in \R^{d_1}\times\R^{d_2}} \|(1+ R\rho(\cdot,y))^\alpha K_F(\cdot,y)\|_1 \leq C_{\alpha,\beta}
\|\delta_{R^2} F \|_{W_\infty^\beta},
\end{equation}
where again $C_{\alpha,\beta}$ does not depend on $R$.
\end{proposition}
\begin{proof}
Note that the heat kernel of the operator $L$ satisfies Gaussian
bounds \eqref{Gauss} so  Proposition \ref{stw} is a straightforward consequence of 
\cite[Lemmas 4.3 and 4.4]{DOS}.
\end{proof}

Recall that the homogeneous dimension of the ambient space is given by  $Q=d_1+3d_2/2 $.

\begin{lemma}\label{lem4.2}
Suppose that $0 \leq \gamma < \min\{d_1/2,d_2/4\}$ and $\beta > Q/2
- \gamma$. For all $y \in \R^{d_1}\times\R^{d_2}$ and $R > 0$,
\begin{equation}\label{eqVB}
\int_{\R^{d_1}\times\R^{d_2}}(1+w_R(x,y))^{-2\gamma} (1 + R \rho(x,y))^{-2\beta}
\,dx \leq C_{\gamma,\beta} |B(y,R^{-1})|.
\end{equation}
Moreover, for all $x,y \in \R^{d_1}\times\R^{d_2}$ and $R > 0$,
\begin{equation}\label{eqDiB}
w_R(x,y) \leq C (1 + R \rho(x,y)).
\end{equation}
\end{lemma}
\begin{proof}
By the homogeneity properties of the distance $\rho$ and the weights
$w_R$, we only prove the case $R=1$. For other case, one just dilate
them by $\delta_t(x',x'')=(tx',t^{3/2}x'')$. By
\eqref{eqCD},
\[\min\{1,|y'|^{-1}\}|x'| \leq 1 + |x' - y'| \leq C (1 + \rho(x,y)),\]
which proves \eqref{eqDiB}.

Because of the translation invariance, to prove \eqref{eqVB}, it is  enough to consider  the case $y''=0$. By
\eqref{eqVr} it suffices to show that 
\[
\int_{\R^{d_1}\times\R^{d_2}}\left(1+\frac{|x'-y'|}{1+|y'|}\right)^{-2\gamma} (1 +
\rho(x,y))^{-2\beta} \,dx \leq C_{\gamma,\beta} (1 + |y'|)^{d_2/2}.
\]
Again we split the integral into two parts, according to the asymptotics
\eqref{eqCD}. In the region $X_1 = \{x \in \R^{d_1}\times\R^{d_2} \colon \, 
 |x''| \geq (|x'| + |y'|)^{3/2}\}$, we choose $\beta_1$ and $ \beta_2$   in such a way that   $\beta = \beta_1
+ \beta_2$,  $\beta_1 > d_1/2-\gamma$ and $\beta_2 > 3d_2/4$.
Then  
\begin{eqnarray*}
\int_{X_1}\left(1+\frac{|x'-y'|}{1+|y'|}\right)^{-2\gamma} (1 +
\rho(x,y))^{-2\beta} \,dx
\\
 \le C(1+|y'|)^{2\gamma} \int_{\R^{d_1}} (1+|x'-y'|)^{-2(\gamma+\beta_1)}
\,dx' \int_{\R^{d_2}} (1+|x''|^{2/3})^{-2\beta_2} \,dx''
\\ \le C_{\gamma,\beta} (1 + |y'|)^{d_2/2}.
\end{eqnarray*}
In the region $X_2 = \{x \in {\R^{d_1}\times\R^{d_2}}\colon \,  |x''| < (|x'| +
|y'|)^{3/2}\}$, instead,  we choose $\beta_1$ and $ \beta_2$   in such a way $\beta =
\tilde\beta_1+\tilde\beta_2$, $\tilde\beta_1 >
d_1/2+d_2/4-\gamma$ and $\tilde\beta_2 > d_2/2$.  Then the integral over
 $X_2$ is estimated by 
\[\begin{split}
\int_{\R^{d_1}\times\R^{d_2}}&\left(1+\frac{|x'-y'|}{1+|y'|}\right)^{-2\gamma} (1+|x'-y'|)^{-2\tilde\beta_1}
\left(1 + \frac{|x''|}{(|x'|+|y'|)^{1/2}}\right)^{-2\tilde\beta_2} \,dx \\
&\leq C_{\gamma,\beta} \int_{\R^{d_1}} \left(1+\frac{|u|}{1+|y'|}\right)^{-2\gamma}
 (1+|u|)^{-2\tilde\beta_1} (|u+y'|+|y'|)^{d_2/2} \,du \\
&\leq C_{\gamma,\beta} \left( (1+|y'|)^{2\gamma} \int_{\R^{d_1}}
(1+|u|)^{-2\nu} \,du + |y'|^{d_2/2} \int_{\R^{d_1}}
(1+|u|)^{-2\tilde\beta_1} \,du \right),
\end{split}\]
where $\nu = \tilde\beta_1 + \gamma - d_2/4 > d_1/2$. The conclusion
follows.
\end{proof}

\begin{proposition}\label{prop4.3}
For all $R > 0$, $\alpha \geq 0$, $\beta > \alpha$, $\gamma \in [0,d_2/4)$, and for all functions $F : \R \to \C$ such that $\supp F
\subseteq \left[ R^2,4R^2 \right]$,
\begin{eqnarray*}
\esssup_{y \in \R^{d_1}\times\R^{d_2}}  |B(y,R^{-1})|^{1/2} \, \| (1 + R\rho(\cdot,y))^\alpha (1+w_R(\cdot,y))^\gamma K_{F(L)}(\cdot,y)\|_2 
\leq C_{\alpha,\beta,\gamma} \|\delta_{R^2} F\|_{W_2^\beta},
\end{eqnarray*}
where the constant $C_{\alpha,\beta,\gamma}$ does not depend on $R$.
\end{proposition}
\begin{proof}
The estimate \eqref{eqSW}, together with \eqref{eqDiB} and a Sobolev embedding,
immediately implies Proposition \ref{prop4.3} in the case $\beta > \alpha + d_2/2 + 1/2$.
On the other hand, in the case $\alpha = 0$, Proposition \ref{prop4.3} follows from 
 Proposition~\ref{prop3.5} for all $\beta > 0$. We obtain now Proposition \ref{prop4.3} for the whole range of 
 exponents  by interpolation
 (see \cite{CJ} and  also \cite[Lemma 4.3]{DOS} for similar methods).
\end{proof}

For the purpose of the next statement we set $D = Q - \min\{d_1,d_2/2\} = \max\{d_1+d_2,3d_2/2\}$.

\begin{coro}\label{coro1}
For all $R > 0$, $\alpha \geq 0$, $\beta > \alpha + D/2$, and for all functions $F : \R \to \C$ such that $\supp F \subseteq \left[ R^2,4R^2 \right]$,
\begin{equation}\label{eqcoro1}
\esssup_{y \in \R^{d_1}\times\R^{d_2}} \| (1 + R\rho(\cdot,y))^\alpha K_{F(L)}(\cdot,y)\|_1
\leq C_{\alpha,\beta} \|\delta_{R^2} F\|_{W_2^\beta},
\end{equation}
where the constant $C_{\alpha,\beta}$ does not depend on $R$. In particular, under the same hypotheses,
\begin{equation}\label{eqcoro1b}
\esssup_{y \in \R^{d_1}\times\R^{d_2}} \int_{{\R^{d_1}\times\R^{d_2}}\setminus B(y,r)} |K_{F(L)}(x,y)| \,dx
\leq C_{\alpha,\beta} (1+rR)^{-\alpha} \|\delta_{R^2} F \|_{W_2^\beta}.
\end{equation}
\end{coro}
\begin{proof}
Corollary \ref{coro1} follows form Proposition~\ref{prop4.3}, together with
\eqref{eqVB} and H\"older's inequality.
\end{proof}

We are finally able to prove our main results.

\begin{proof}[Proofs of Theorems~\ref{thm1m} and \ref{thm2r} ]
To prove Theorem~\ref{thm1m} We can follow the lines of the proof of \cite[Theorem 3.1]{DOS},
where the inequality (4.18) there is replaced by our
\eqref{eqcoro1b}. Next We can use that same argument as in  \cite[Section 6]{DOS}
to conclude the proof of Theorem~\ref{thm2r}, see also \cite{MaS12}. 
\end{proof}

\section{Final remarks}\label{sec5}

The natural open problem related to the sharp spectral multiplier results which we prove in this paper is to extend them 
to the class of all  operators $L_\sigma$ defined by \eqref{eq1} for  $\sigma >0$. Another interesting problem which arises
is to obtain possible precise description of the spectral decompositions of operators $L_\sigma$. 

Now we shall  show that, if $d_1 \ge d_2/2$, then the result in Theorem~\ref{thm1m} is sharp. More precisely, if $d_1 \ge d_2/2$ and $s<D/2=(d_1+d_2)/2$, then the  weak type $(1,1)$ estimates in Theorem~\ref{thm1m} cannot hold. Indeed, if we consider the functions
$H_t(\lambda)=\lambda^{it}$, then, for $t>1$, and any $\eta\in C^\infty_c(\R_+)$
\[
 \|\eta H_t \|_{W_2^s} \sim t^s.
\]
On the other hand, we make the following observation.
\begin{proposition}
Suppose that $L$ is the Grushin operator acting on $X = \R^{d_1}\times \R^{d_2}$.
Then the following lower bounds holds:
\[
\|H_t(L)\|_{L^1\to L^{1,w}}=\|L^{it}\|_{L^1\to L^{1,w}} \ge C (1+|t|)^{(d_1+d_2)/2}
\]
for all $t>0$.
\end{proposition}
\begin{proof}
Because the Grushin operator is elliptic on $X_0 = \{ x \in {\R^{d_1}\times\R^{d_2}}\colon \,  x' \neq 0\}$, one can use the same argument as in \cite{SW} to prove that, for all $y \in X_0$,
\[
|p_t(x,y) - |y'|^{- d_2}(4\pi
t)^{-(d_1+d_2)/2}{\rm e}^{-\rho(x,y)^2/4t}|
 \le C
t^{1/2}t^{-(d_1+d_2)/2}
\]
for all $x$ in a small neighbourhood of $y$ and all $t \in (0,1)$. Here $p_t=K_{\exp(-tL)}$ is the heat kernel corresponding to the Grushin operator. The rest of the argument is the same as in \cite{SW}, so we skip it here.
\end{proof}

To show that  Theorems~\ref{thm1m} and \ref{thm2r} are sharp one can also use the results described 
in \cite{Ke}.

 \bigskip

\noindent \noindent {\bf Acknowledgements:} This project was
supported by Australian Research Council Discovery grants
DP110102488.

\end{document}